\documentclass[12 pt,reqno,oneside]{amsart}
\usepackage{amsmath, amsfonts, amssymb, euscript, amscd, latexsym, mathrsfs, bm, color, bbm, relsize, upgreek, xfrac, tabularx, float}
 \usepackage{titlesec}
 \usepackage[pdftex,colorlinks=true,
                       pdfstartview=FitV,
                       linkcolor=blue,
                       citecolor=blue,
                       urlcolor=blue,
           ]{hyperref}
            \usepackage{enumitem}
\usepackage{pstricks}
\usepackage{amsthm}

\newtheorem{theorem}{Theorem}[section]
\newtheorem{corollary}[theorem]{Corollary}
\newtheorem{lemma}[theorem]{Lemma}
\newtheorem{proposition}[theorem]{Proposition}
\newtheorem{remark}[theorem]{Remark}
\newtheorem{definition}[theorem]{Definition}

\def\aa#1{ \begin{align*} #1 \end{align*} }
\def\aaa#1{ \begin{align} #1 \end{align} }
\def\mm#1{ \begin{multline*} #1 \end{multline*} }
\def\mmm#1{ \begin{multline} #1 \end{multline} }

\def\begeq{\begin{equation} \begin{cases}} 
\def\endeq{ \end{cases} \end{equation}}
\def\eq#1{ \begeq #1 \endeq }

\def\bege{\begin{equation*} \begin{cases}} 
\def\ende{ \end{cases} \end{equation*}}

\def\eqn#1{ \begin{equation} \begin{aligned} #1 \end{aligned}\end{equation}}
\def\eqm#1{ \begin{equation*} \begin{aligned} #1 \end{aligned}\end{equation*}}

\newcommand{\eps}{\varepsilon}
\newcommand{\pl}{\partial}
\newcommand{\gt}{\geqslant}
\newcommand{\lt}{ \leqslant}
\newcommand{\te}{\theta}
\newcommand{\sub}{\subset}
\newcommand{\dl}{\updelta}
\newcommand{\al}{\alpha}
\newcommand{\gm}{\gamma}
 \newcommand{\Gm}{\Gamma}

 \newcommand{\la}{\lambda}
 \newcommand{\La}{\Lambda}
 \newcommand{\sg}{\sigma}

\newcommand{\mc}{\mathcal}

\newcommand{\mf}{\mathfrak}

\newcommand{\C}{{\rm C}}
\newcommand{\td}{\tilde}

\newcommand{\teta}{\upvartheta}

\newcommand{\f}{{\sss (\sfrac14)}}

\DeclareFontFamily{U}{solomos}{}
\DeclareErrorFont{U}{solomos}{m}{n}{10}
\DeclareFontShape{U}{solomos}{m}{n}{
  <-> s*[1.1]  gsolomos8r
}{}

\newcommand{\cp}{\text{\usefont{U}{solomos}{m}{n}\symbol{'153}}}
\newcommand{\X}{{\scriptstyle {\mathscr X}}}

\newcommand{\kp}{\kappa}
\newcommand{\T}{{\sss T}}

\newcommand{\x}{\times}
\newcommand{\mto}{\mapsto}

\newcommand{\rf}{\eqref}
\newcommand{\bi}{\begin{itemize}}
\newcommand{\ei}{\end{itemize}}
\newcommand{\be}{\begin{enumerate}}
\newcommand{\ee}{\end{enumerate}}

\DeclareMathOperator{\ind}{\mathbbm{1}}

\newcommand{\lb}{\label}

\newcommand{\fdot}{\,\cdot\,}

\newcommand{\sss}{\scriptscriptstyle}

\newcommand{\sig}{\varsigma}

\newcommand{\ro}{\varrho}

\def\Rnu{{\mathbb R}}

\def\Nnu{{\mathbb N}}

\def\ffi{\varphi}

\renewcommand{\l}{\left}
\renewcommand{\r}{\right}

\newcommand{\ou}{\bar{u}}
\newcommand{\uu}{\underline{u}}

\long\def\symbolfootnote[#1]#2{\begingroup%
\def\thefootnote{\fnsymbol{footnote}}\footnote[#1]{#2}\endgroup}
\titleformat{\section}[hang]{\large\bfseries}{\thesection.}{1ex}{}{}
\titleformat{\subsection}[hang]{\normalsize\bfseries}{\thesubsection}{2ex}{}{}
\titleformat{\subsubsection}[hang]{\small\bfseries}{\thesubsubsection}{2ex}{}{}

\include{srctex.sty}

\usepackage{lipsum}

\usepackage{algorithmic}

\title[Touchdown solutions in general MEMS models]
{Touchdown solutions in general MEMS models}

\author[R. Clemente]{Rodrigo Clemente}
\author[J.M. do \'O]{Jo\~ao Marcos do \'O$^*$}
\author[E.\ da Silva]{Esteban da Silva}
\author[E.\ Shamarova]{Evelina Shamarova}

\address[R. Clemente]{Departamento de Matem\'atica, 
Universidade Federal Rural de Pernambuco,
52171-900, Recife, Pernambuco, Brazil}
\email{\href{mailto:rodrigo.clemente@ufrpe.br}{rodrigo.clemente@ufrpe.br}}

\address[$^*$Corresponding author: J.M. do \'O ]{Departamento de Matem\'atica, 
Universidade Federal da Para\'iba, 
58051-900, Jo\~ao Pessoa, Brazil}
\email{\href{mailto:jmbo@pq.cnpq.br}{jmbo@pq.cnpq.br}}

\address[E.\ da Silva]{Departamento de Matem\'atica, 
	Universidade Federal do Rio Grande do Norte, 
	59078-970, Natal, Brazil}
\email{\href{mailto:esteban.silva@ufrn.br}{esteban.silva@ufrn.br}}

\address[E.\ Shamarova]{Departamento de Matem\'atica, 
	Universidade Federal da Para\'iba, 
	\newline\indent
	58051-900, Jo\~ao Pessoa, Brazil}
\email{\href{mailto:evelina.shamarova@academico.ufpb.br}{evelina.shamarova@academico.ufpb.br}}

\begin{document}

\begin{abstract}
We study general equations modeling electrostatic MEMS devices
\begin{equation} \begin{cases}
 \label{P}
\varphi\big(r,- u'(r)\big)=\lambda\int_0^r\frac{f(s)}{g(u(s))}\,\mathrm{d}s, 
& r\in(0,1), \\
0  < u(r) < 1,  & r\in(0,1), \\
u(1) = 0,
\tag{$P_\la$}
 \end{cases} \end{equation}
where $\varphi$, $g$, $f$ are some functions on $[0,1]$ and $\lambda>0$ is a parameter. We obtain results on the existence and regularity of a touchdown solution to \eqref{P} and
find upper and lower bounds on the respective pull-in voltage.
In the particular case, when $\varphi(r,v) = r^\alpha |v|^\beta v$, i.e., when the associated
differential equation involves the operator $r^{-\gamma}(r^\alpha |u'|^\beta u')'$,
we obtain an exact asymptotic behavior of the touchdown solution in a neighborhood
of the origin.
\end{abstract}

\vspace{4mm}

\maketitle

{\footnotesize
\vspace{1mm}

}



\section{Introduction}
Microelectromechanical systems (MEMS) 
 are microscopic devices consisting of electrical and mechanical components 
 built together on a chip. This system  is usually of a tiny dimension, between 1 and 100 micrometers. MEMS devices are indispensable  in the modern technology 
such as telecommunications,  biomedical engineering, exporing the space, etc. For more information on the applications of MEMS, we refer the reader to  P.~Esposito {et al.} \!\cite{MR2604963} and to D. Bernstein, J. Pelesko \cite{MR1955412}.

Equations describing MEMS are written with respect to the displacement of 
a thin and deformable microplate (membrane) 
which is fixed along the boundary of a bounded domain
(usually a unit ball)  below a fixed rigid plate.  
 The membrane deflects towards the rigid plate when one applies a voltage, 
represented here by $\la$.
 It may happen that the membrane touches the rigid plate when the voltage reaches
 a certain  critical value $\la^*$\!, called the pull-in voltage. In the latter case,
 an instability is created within the system which may greatly affect the MEMS device.
 The corresponding solution for the model is then called
a touchdown solution. Since one aims to achieve a better design for a MEMS device,
studying touchdown solutions in different MEMS models is therefore of great
interest and importance.


\subsection{Motivation and related results}
 Recently, many authors proposed MEMS models generalizing the classical one,
 i.e., whose main equation is  $-\Delta u = \frac{\lambda f(x)}{(1-u)^2}$.
 We refer the reader to \cite{MR2604963,MR1955412} for a survey of results
 on the classical MEMS model.
 Several authors, e.g., \cite{MR2500847,MR2471138, KMS, Korman},  studied 
 so-called $p$-MEMS models,
 that is,  when the main MEMS equation contains a $p$-Laplacian
 instead of the Laplacian. 
Subsequently, MEMS problems on a ball involving a more general radial nonlinear 
 operator $L(\al,\beta,\gm)[u] = r^{-\gm} (r^\al |u'|^\beta u')'$
 received some  attention in the literature \cite{MR3759060,MR3491876,MR3571913}.
 Problems containing the operator $L(\al,\beta,\gm)$ have been also considered 
 in many other contexts \cite{MR1422009, MR1929156, jmss}. As such, in \cite{MR1422009}, Cl\'ement {et al.} studied a Brezis-Nirenberg-type problem, in \cite{MR1929156}, J.~Jacobsen and K.~Schmitt studied a Liouville-Bratu-Gelfand problem. We refer the reader to \cite{MR2569324,MR2166492} 
 for a survey of models containing the operator $L(\al,\beta,\gm)$.
 On the other hand, touchdown issues in MEMS models received a lot of attention in the literature; see, e.g., \cite{Guo, Guo1, Guo2, GM, KMS}.

Motivated by the aforementioned works and in the spirit of the pioneering work \cite{MR0382848}, 
we study here  the very general class of MEMS problems 
described by its unique formulation \rf{P}. 
Our main  motivation for the 
choice of the model \rf{P} is to encompass the following below 
important examles, where the function $h(r)$ stands for the 
spatially varying dielectric permittivity, which is assumed positive
and integrable on $[0,1]$.

\begin{enumerate}[label=($E_\arabic*$)]
	
\item \lb{E1} $L(\al,\beta,\gm) = r^{-\gm}(r^{\alpha}|u'|^{\beta}u')'$
with $\al\lt \gm$ and $\beta>-1$; 
the associated equation in \rf{P} takes the form
\aaa{
\lb{G}
-r^\alpha|u'|^\beta u'= \la \int_0^r\frac{s^\gm h(s)}{g(u)}\, ds;
\qquad \ffi(r,v)=r^\alpha|v|^\beta v.
\tag{$\hat P_\la$}
}
\item \lb{E2} $\sum_{i=1}^n L(\al_i,\beta_i,\gm)$ with 
$\al_i\lt \gm$ and $\beta_i>-1$. In this case, the 
associated equation in  \eqref{P} becomes
\aa{
-\sum_{i=1}^n r^{\alpha_i}|u'|^{\beta_i} u'
=\la \int_0^r\frac{s^\gm h(s)}{g(u)}\, ds; \qquad
\ffi(r,v)=\sum_{i=1}^n r^{\alpha_i}|v|^{\beta_i} v.
}
\item \lb{E3} {\it Radial form of the $p(x)$-Laplacian.} For this operator, the associated 
equation in \eqref{P} is
\aa{
-r^{N-1}|u'|^{p(r)-2} u'=\la \int_0^r \frac{s^{N-1}h(s)}{g(u)}\, ds; 
\qquad  \ffi(r,v)= r^{N-1}|v|^{p(r)-2} v,
}
where $p(r)$ is a continuous function with values in $[1+\epsilon,N)$,
$\epsilon \in (0,1)$, $N>1$.
\item \lb{E4}
{\it Laplace-Beltrami operator on an $N$-dimensional sphere
of radius $\rho\gt 1$.} Here, the equation in \rf{P} is
\aa{
-\rho\l(\sin\l(\frac{r}{\rho}\r)\r)^{\!N-1}\!\! u'=\la \int_0^r  \frac{(\sin(\frac{s}{\rho}))^{N-1}h(s)}{g(u)}\, ds; 
\qquad \ffi(r,v)= \rho \sin\l(\frac{r}{\rho}\r)^{N-1}\!\! v.
}

\end{enumerate}

We remark that apart from examples \ref{E1}--\ref{E4}, our general formulation  \rf{P} can possibly be applied to other MEMS models.

Note that the operator $L(\al,\beta, \gm)$ includes a $k$-Hessian and 
a $p$-Laplacian. Namely, the two latter cases are associated with the 
following values of $\al$, $\beta$, $\gm$:
\begin{table}[htbp]
\centering
\begin{tabularx}{0.55\textwidth}{X  X  X  X  X X} 
\hline
 & & $\al$ & $\beta$ & $\gm$ & $\te$ \\
\hline
\text{$k$-Hessian} & & $N-k$ & $k-1$ & $N-1$& $2k$ \\
\hline
\text{$p$-Laplacian} &  & $N-1$ & $p-2$ & $N-1$ & $p$ \\
\hline
\end{tabularx}
\label{table:tab1} 
\end{table}

The number $\te = \gm +2+ \beta - \al$ will be used in Theorem \ref{ss2222}, 
and $N$ stands for the spatial dimension of the associated problem.

Model \ref{E1}, in application to MEMS, 
has been previously studied in the literature by some authors
\cite{GM,MR3491876, MR3571913}.
The role of the parameters $\al$, $\beta$, $\gm$ is to describe  
properties of the substance between the plates of the MEMS device.
If this substance is inhomogeneous, i.e., its properties are changing
from layer to layer, a better model is \ref{E2}. To describe properties
of a substance changing along the radius of the ball, one uses 
model \ref{E3}. Finally, model \ref{E4} was proposed in 
\cite{MR3759060} to account for the specific profile
of the deflected membrane.

\subsection{Hypotheses}\label{hypo}
Throughout Sections \ref{s2}--\ref{s4}, we assume  

\be [label=($H_\arabic*$)]
\item  $g\in  \C\big((-\infty,1]\to [0,\infty)\big)$ is strictly decreasing 
and $g(1)=0$. \label{a1111} \lb{H1}
\item \lb{H4} $f: [0,1]\to [0,+\infty)$ is a measurable function such that for
its primitive function 
\aa{
F(r) = \int_0^r f(s)\, ds,
}
it holds that $F(1)<+\infty$ and $F(r) > 0$ for all $r\in (0,1]$.
\item \lb{H2} $\ffi:[0,1]\x \Rnu \to \Rnu$  is a measurable function with the following properties.
For each $r\in (0,1]$, the function $v\mto \ffi(r,v)$ is strictly increasing,
continuous, and such that $\ffi(r,0) = 0$, $\lim_{v\to+\infty} \ffi(r,v) = +\infty$,
and $\int_0^1 \ffi^{-1}(s, F(s))\, ds < +\infty$, 
where for each fixed $r\in [0,1]$, $\ffi^{-1}(r,\fdot)$ denotes the inverse to 
$\ffi(r,\fdot)$.
\item \lb{H3} There exist  a polynomial-like function $\mc P(v)=\sum_{i=1}^m c_i v^{d_i}$, where 
$c_i,d_i\in (0,+\infty)$, and 
a measurable function $a: [0,1]\to (0,+\infty)$ 
with the property $\sup_{r\in [0,1]} a(r) <+\infty$, 
such that for all $(r,v)\in [0,1]\x [0,+\infty)$,
\aa{
\ffi(r,v) \lt a(r) \mc P(v).
}
\vspace{-7mm}
\ee
For the problem \rf{G}, our goal is to construct a touchdown solution and 
obtain its exact asymptotic behavior in a neighborhood of the origin.
For this task, we assume
\be[label=($A_\arabic*$)]
\item \ref{a1111} holds; \label{A1}
\item  \label{A2} $g\in \C^2(0,1)$ and there exist constants $q\in (0,1)$ 
and $A>0$ such that 
\aa{
g'(u) = - A\, g(u)^q (1+o(1)) \qquad (u\to 1_-).
}
\vspace{-6mm}
\item $\beta> -1$, $\al>\beta+1$, $\gm\gt \al$.
\item $h: \Rnu_+\to [c,b]\sub (0,+\infty)$ is a measurable function such that 
$h(r) = C(1+o(1))$ ($r\to 0_+$), 
where $C>0$ is a constant.
\label{A4} 

\ee
\begin{remark}
\rm
Note that \ref{A1} and \ref{A2} hold for the MEMS nonlinearity
$g(u) = (1-u)^p$, $p>1$. In this case, $A=p$ and $q=\frac{p-1}{p}$.
Other examples may include the nonlinearity of the form
$g(u) = \zeta(u)^p$, $p>1$, where $\zeta\in \C^2(0,1)$ 
satisfies \ref{A1}
and in a neighborhood of $u=1$ solves the equation 
$c_1 \zeta(u)^{\nu +1} - \zeta(u) = c_2(u-1)$ with $c_1,c_2,\nu>0$.
\end{remark}
\begin{remark}
\lb{a2222}
\rm
Note that Assumption \ref{A2} is equivalent to the following
\aa{
g'(g^{-1}(w)) = - A \, w^q (1+o(1)) \qquad (w\to 0_+).
}
\end{remark}

\subsection{Contributions of this work}
In this paper, we address the question of existence of a pull-in voltage and a touchdown
solution for the MEMS model \rf{P} encompassing \ref{E1}-\ref{E4}.
 Furthermore, we address obtaining an exact
asymptotic behavior, in a neighborhood of the origin,  of the touchdown solution for model \ref{E1}.

Most of our results are obtained under hypotheses \ref{H1}--\ref{H3}.
More specifically, we start by showing the existence of a finite pull-in voltage 
$\la^*$ employing the sub- and supersolution method. 
We then explicitely compute upper and lower bounds on $\la^*$. 
Since most of standard methods cannot be applied to our very general framework,
we use a different approach, relying on the Zorn lemma, to show
the existence and monotonicity (in $\la$) of the maximal and minimal branches 
of solutions to problem \rf{P}. Finally, assuming the existence of the limit
$\lim_{\la\to\la^*} u_\la(0) = 1$ for the maximal or minimal branches,
we obtain the existence and uniqueness of a touchdown solution
to \rf{P}. In addition, we obtain the $\C^2$-regularity
of the touchdown solution under the assumption of the $\C^1$-regularity of
the function $\ffi$.

Our most challenging result is obtaining a sharp asymptotic representation
(in a neighborhood of the origin)
of the touchdown solution for the model \ref{E1}.
A related result was obtained in \cite{GM}; however, only in the form $u^*(r) = G^{-1}\big[\frac{r^\teta}{\mc A}(1+o(1))\big]$
with some positive constants $\teta$ and $\mc A$, and the function
$G(u) = \int_0^u g(s)^{\frac1{\beta+1}} ds$. An exact asymptotic behavior,
 similar to as it is announced  in Theorem \ref{ss2222}
of the present work,
was only obtained in \cite{GM} for a few explicitely given functions $g$. 
In contrast, the result of our Theorem \ref{ss2222} does not rely on any  
specification of the function $g$; it is obtained for an arbitrary function
$g$ satisfying \ref{A1} and \ref{A2}.

We believe that some of our results, such as upper and lower bounds on the pull-in voltage
 and the sharp asymptotic behavior of the touchdown solution, can be used
 for studying MEMS by physicists, engineers, and   numerical mathematicians.
 In addition, we present new methods, e.g.,  the use of the Zorn lemma to 
 show the existence of the maximal
 and minimal branches (Theorem \ref{prop1111})
  or fitting unknown parameters in the representation of a singular solution
   to make a certain map have a fixed point 
  (Theorem \ref{ss2222}).
Below, we announce the main results of this work.

\begin{theorem}[Finite pull-in voltage and its bounds]
\lb{t-pulin}
Assume \ref{H1}--\ref{H3}. Then, the pull-in voltage $\la^*$ for problem \rf{P} exists;
furthermore, it is positive and finite. Moreover, we have the upper bound
 \aaa{
 \lb{pulin}
 \la^* 
 \lt p_{d,e}^{-1}\Big( 
 2 P_{d,e}\big(\mc P(1)\sup_{r\in [0,1]}a(r)\big)
 P_{d,e}(g(0)) \big[p_{d,e} \l(F(\sfrac12)\r) \big]^{-1} \Big),
 }
 where $F(\fdot)$ is defined in \ref{H4}, 
 $P_{d,e}(r) = \max\{r^\frac1d, r^\frac1e\}$,
$p_{d,e} (r) = \min\{r^{\frac1d}, r^{\frac1e}\}$
with $d = \max_i d_i$ and $e=\min_i d_i$.
Furthermore, defining $\Phi(\la) = \int_0^1 \ffi^{-1}(s,\la \,F(s)) ds$,
we have the lower bound
\aaa{
\lb{l-pulin}
\la^*\gt 
\begin{cases}
g(\Phi(1)) \;\; \text{if} \;\; \Phi(1)<1,\\
\sup_{\dl\in [0,1]}\Phi^{-1}(\dl)g(\dl)
\;\; \text{if} \;\; \Phi(1)\gt 1.
\end{cases}
}

\begin{theorem}
\lb{prop1111}
Assume \ref{H1}--\ref{H3}. Then, 
there exist a unique maximal and a unique minimal branches of solutions 
for problem \rf{P}, where $\la$ varies over $(0,\la^*)$.
\end{theorem}
 \end{theorem}
 \begin{theorem}[Existence of a touchdown solution]
\lb{th2222}
Assume \ref{H1}--\ref{H3}.
Let $\{u_\la\}_{\la\in (0,\la^*)}$ be the maximal or minimal branch of solutions to \rf{P} such that
$\lim_{\la\to\la^*} u_\la(0) = 1$. Then, the pointwise
 limit $u_{\la^*}(r) = \lim_{\la\to\la^*}u_\la(r)$ exists and it is an integral touchdown solution
 to $(\rm P_{\!\la^*})$.
\end{theorem}
\begin{theorem}[Uniqueness of a touchdown solution]
\lb{unique2222}
Assume \ref{H1}--\ref{H3}. Then, an integral touchdown solution, whenever it exists, is unique
and coincides with $u_{\la^*}$, the pointwise limit of the maximal branch
as $\la\to\la^*$.
\end{theorem}
\begin{theorem}[Asymptotic behavior of the touchdown solution for problem \rf{G}] 
\lb{ss2222}
Assume \ref{A1}--\ref{A4}. Then, the pull-in voltage $\la^*$ 
and the touchdown solution $u^*$ for problem \rf{G} 
exist. Moreover, as $r\to 0_+$, 
it holds that
\aaa{
& u^*(r) = 1 - \mc C r^{\te\cp(1-q)} + o(r^{\te\cp(1-q)}), \lb{s-sol1111} \\
& \text{where} \;\;
\cp = \l((1-q)(\beta+1)+1\r)^{-1}, \quad
\kp = \l[\frac{A}{\te\cp}\r]^{\!(\beta+1)\cp}\hspace{-1mm}
 \l[\frac{C}{\gm+1-\te\cp}\r]^{\!\cp}, \lb{co1111}
}
and $\mc C =  \kp^{1-q}(\la^*)^{\cp(1-q)}A^{-1}(1-q)^{-1}$.
\end{theorem}
The idea of the proof of Theorem \ref{ss2222} is as follows.
First, we transform problem \rf{G} to an equivalent one, but given 
in a neighborhood of infinity $[T,+\infty)$.
We then find a certain normed space $X_T$
 and a compact continuous map on it to apply
 Schauder's fixed point theorem. 
The constants $\cp$ and $\kp$ in Theorem \ref{ss2222} are uniquely fixed
to condition a fixed point map take values in $X_T$.

We remark that our approach to obtaining
the existence and asymptotics of the touchdown solution to \rf{G}
differs significantly from the method used in \cite{GM}. It is somewhat in the spirit
of the method of \cite{jmss}, although there is no technical simularities. 
\subsection{Outline} 
In Section \ref{s2},
we present a sub- and supersolution method appropriate
for our general framework. In the same section, we prove the existence
of a pull-in voltage $\la^*$ for the problem \rf{P} and compute upper and lower
bounds on $\la^*$. In Section \ref{s3}, we obtain the existence
of the maximal and minimal branches $u_\la$ of solutions to \rf{P}
and the monotonicity of $u_\la$ in $\la$. Furthermore, we obtain 
the existence and uniqueness of a touchdown solution to problem \rf{P}. 
In Section \ref{s4}, we obtain the $\C^2$-regulatity of integral solutions,
and, in particular, of the touchdown solution, to \rf{P}. Finally, Section \ref{s5}
is dedicated to obtaining the asymptotic representation \rf{s-sol1111} of the touchdown
solution to problem \rf{G}.

\section{Finite pull-in voltage for Problem \rf{P}}
\lb{s2}

\subsection{Preliminaries}

In this subsection, we give definitions of integral solutions, sub- and supersolutions, 
maximal and minimal solutions which will be used throughout the paper.
\begin{definition}[Integral solution]
\lb{int-sol1111}
We say that $u\in \C^1(0,1)\cap \C[0,1]$ 
is an integral solution to problem \eqref{P} if 
$\int_0^r\frac{f(s)}{g(u(s))}\,\mathrm{d}s <+\infty$ and
$u(\fdot)$ solves \rf{P}.

\end{definition}
\begin{definition}[Sub- and supersolutions]
\lb{int-subsuper1111} 
We say that a measurable function  $\ou: [0,1]\to [0,1]$ 
is a supersolution to problem \eqref{P} 
 if $\ou(1)\gt 0$, $\int_0^1\frac{f(r)}{g(\ou(r))}\,\mathrm{d}r <+\infty$,  and
\aaa{
\lb{super2222}
\ou(r) \gt  \int_r^1\ffi^{-1}\left(t, \la\int_0^t\frac{f(s)}{g(\ou(s))} ds\right) dt.
}
Furthermore, we say that a measurable function  $\uu: [0,1]\to [0,1]$ is 
a subsolution to problem \eqref{P} 
if for some $r_0\in (0,1]$, $\uu(r)= 0$ on $[r_0,1]$, 
$\int_0^1\frac{f(r)}{g(\uu(r))}\,\mathrm{d}r <+\infty$,  and
\aaa{
\lb{sub2222}
\uu(r)  \lt \int_r^1\ffi^{-1}\left(t, \la \int_0^t \frac{f(s)}{g(\uu(s))} ds \right) dt.
}
\end{definition}	
\begin{definition}[Minimal or maximal integral solution]
 We call a function $u\in \C^1(0,1)\cap \C[0,1]$ a
 minimal (maximal) integral solution to problem \rf{P} if it is an integral
 solution solution to \rf{P} in the sense of Definition \ref{int-sol1111} and
 for any other integral solution $v$ to \rf{P} it holds that $u(r) \lt v(r)$ (resp. $u(r)\gt v(r)$)
for all $r\in [0,1]$.
 \end{definition}
 \begin{definition}
The value
\aa{
\la^* = \sup\{\la>0:   \;  \text{\rf{P} possesses an integral solution}\}.
}
is called a pull-in voltage for problem \rf{P}.
\end{definition}
\begin{definition}
An integral solution $u$ to \rf{P} is called a touchdown solution if
$u(0) = 1$.
\end{definition}

\subsection{Verification of hypotheses \ref{H1}--\ref{H3} for models \ref{E1}--\ref{E4}}
Here, we will verify assumptions \ref{H2} and \ref{H3} for the functions
$\ffi$ involved in the models  \ref{E1}--\ref{E4}. 
First, we verify \ref{H3}. 
Let $(r,v)\in [0,1]\x [0,+\infty)$.
In \ref{E1}, one defines $a(r)=r^\alpha$ and $\mc P(v) = v^{\beta+1}$,
so $\varphi(r,v) = a(r) \mc P(v)$. 
In \ref{E2}, 
 $\varphi(r,v)\lt r^{\tilde{\alpha}} \sum_{i=1}^{n} v^{\beta_i+1}$, where $\tilde{\alpha}=\min_{i=1,...,n}\alpha_i$; therefore, $a(r)=r^{\tilde{\alpha}}$ 
 and $\mc P(v)=\sum_{i=1}^{n} v^{\beta_i+1}$. 
 Further, in \ref{E3}, 
 $\varphi(r,v) \lt r^{N-1} (v^{\bar p -1} + v^{\td p-1})$, where $\tilde{p}=\min_{r\in[0,1]}p(r)$
 and $\bar p = \max_{r\in[0,1]}p(r)$,
 so that we can set $a(r)=r^{N-1}$ and  $\mc P(v)= v^{\bar p -1} + v^{\td p-1}$. 
 Finally, in \ref{E4}, 
 $\varphi(r,v) \lt \rho(\sin\!\big(\frac{r}{\rho}\big))^{N-1} v $, so one defines 
 $a(r)=\rho(\sin\! \big(\frac{r}{\rho}\big))^{N-1}$ and $\mc P(v)= v$.
 Note that in all cases, a verification of the property $\sup_{r\in [0,1]} a(r) <+\infty$
  is straightforward. 
 
 To verify \ref{H2}, we only have to check that
$\int_0^1 \ffi^{-1}(r, F(r))\, dr < +\infty$ since the rest of the conditions in \ref{H2}
is obviously fulfilled for all four models. 
 To this end, we will show that for all $r\in (0,1]$, 
$ \ffi^{-1}(r, F(r)) \lt G\big(\int_0^r h(s) ds\big)$
in each of the cases \ref{E1}--\ref{E4}, where
$G$ is an increasing function.
As before, we let
$(r,v)\in [0,1]\x [0,+\infty)$. In the example \ref{E2}, 
 (we consider \ref{E1} as a particular case of \ref{E2}), we have that
$\ffi(r, v)  \gt r^{\bar \al} \sum_{i=1}^{n} v^{\beta_i+1} \gt r^{\bar \al} v^{\beta_1+1}$,
where $\bar \al=\max_{i=1,...,n}\alpha_i$.
We also note that $\bar \al\lt \gm$. Therefore,
\eqm{
&\ffi^{-1}(r,F(r))\lt \big(r^{-\bar \al} F(r)\big)^\frac1{\beta_1+1}
\lt \Big(\int_0^r h(s) ds\Big)^\frac1{\beta_1+1} \quad \text{in \ref{E1} and \ref{E2}};\\
&\ffi^{-1}(r,F(r))\lt \Big(\int_0^r h(s) ds\Big)^\frac1{p(r)-1} 
\lt 
 \Big(\int_0^r h(s) ds\Big)^\frac1{\epsilon} + 
 \Big(\int_0^r h(s) ds\Big)^\frac1{N-1} 
\quad \text{in \ref{E3}}; \\
&\ffi^{-1}(r,F(r))\lt const \int_0^r h(s) ds
\quad \text{in \ref{E4}}.
}

\subsection{Sub- and supersolution method}

Here we develop a suitable sub- and supersolution method for problem \eqref{P}.
We will show that \rf{P}
 admits an integral solution whenever $\la\in (0,\la^*]$ and has no integral 
 solution if $\la>\la^*$.

\begin{proposition}[Sub- and supersolution method]
\label{tss}
Assume that \ref{H1}--\ref{H3} hold. Further 	
 assume there exist $\uu$ and $\ou$, an integral subsolution and an integral supersolution, respectively, to problem \eqref{P} such that $\uu\lt \ou$. 
Then, there exists an integral solution $u(\fdot)$ to \eqref{P} satisfying
\aa{
\uu  \lt u \lt \ou.
}
Furthermore, $u(\fdot)$ can be represented as
\aaa{
\lb{limit1111}
u(r) =\lim_{k\to\infty} u_k(r),
}
where $\{u_k\}$ is a sequence given recursively as follows: 
$u_0=\ou$;
$u_{k+1}$ is defined via $u_k$ by the relation
\aaa{
\lb{uk1111}
u_{k+1}(r) =\int_r^1\ffi^{-1}\left(t,\la \int_0^t \frac{f(s)}{g(u_k(s))}
 ds\right) dt, \quad r\in [0,1].
}	
\end{proposition}
\begin{proof}
First, we note that the equation in \rf{P} can be rewritten as follows:
\aaa{
\lb{eq1111}
u(r) =\int_r^1\ffi^{-1}\left(t, \la \int_0^t \frac{f(s)}{g(u(s))} ds\right) dt, \quad r\in [0,1].
}
In what follows, for simplicity of notations, we define
 \aa{
 \chi(r,u) = \la\, \frac{f(r)}{g(u)}.
 }
{\it Step 1. The sequence $\{u_k\}$ is well-defined and $\uu \lt u_k \lt \ou$.}
We show the following two inequalities at the same time:
  $\int_0^1 \chi(s,u_k(s))<\infty$ and $\uu \lt u_k \lt \ou$ for all $k$. 
When $k=0$, it follows from the
definition of a subsolution and supersolution and from the fact that $u_0 = \ou$. Suppose,
as the induction hypothesis, that
$\int_0^1 \chi(s,u_k(s))<\infty$ and  $\uu \lt u_k \lt \ou$. Since 
$\chi(s,u)$ is increasing in $u$, we have that 
 \aaa{
 \lb{h1111}
 \chi(r,\uu(r)) \lt \chi(r,u_k(r)) \lt \chi(r,\ou(r)).
 } 
 Therefore, by \rf{super2222} and \rf{sub2222},
 \mm{
 \uu \lt 
 \int_r^1\ffi^{-1}\left(t, \int_0^t \chi(r,\uu(r)) ds\right) dt  \lt 
 \int_r^1\ffi^{-1}\left(t, \int_0^t \chi(s,u_k(s)) ds\right) dt = u_{k+1}(r) \\
 \lt \int_r^1\ffi^{-1}\left(t, \int_0^t \chi(r,\ou(r)) ds\right) dt \lt \ou.
 }
 This implies that $\chi(r,\uu(r)) \lt \chi(r,u_{k+1}(r)) \lt \chi(r,\ou(r))$
 and $\int_0^1 \chi(s,u_{k+1}(s))<\infty$. 
 
 {\it Step 2. $u_{k+1}(r) \lt u_k(r)$ for all $k$ and $r\in [0,1]$.}
 As before, we show this affirmation  by induction on $k$. When $k=0$,
 by Step 1, we have $u_0(r) = \ou(r) \gt u_1(r)$.
 Suppose we have proved that $u_{k}\lt u_{k-1}$. This implies that 
 $\chi(s,u_{k}(s)) \lt \chi(s,u_{k-1}(s))$. Therefore,
 \aa{
 u_{k+1}(r) = \int_r^1\ffi^{-1}\left(t, \int_0^t \chi(s,u_{k}(s)) ds\right)
 \lt \int_r^1\ffi^{-1}\left(t, \int_0^t \chi(s,u_{k-1}(s)) ds\right) = u_{k}(r).
 } 

{\it Step 3. The pointwise limit \rf{limit1111} solves \rf{eq1111}.}
Since $u_k(r)$ is non-increasing in $k$ and $\uu \lt u_k \lt \ou$, there 
exists a finite pointwise limit $u(r)$ given by \rf{limit1111}. Let us show 
that $u(r)$ solves \rf{eq1111}. Indeed, by \rf{h1111} and the continuity
of $\ffi^{-1}$, we can pass
to the limit in \rf{uk1111} by the dominated convergence theorem.

Since $u(s)\lt\ou$, we have that $0\lt \int_0^t \chi(s,u(s)) ds <\infty$. Therefore,
we can differentiate \rf{eq1111} with respect to $r$. Applying 
$\ffi(r,\fdot)$ to the both sides of the resulting equation, we obtain
the equation in \rf{P}. Thus, we conclude that $u$ is an integral solution to \eqref{P}.
\end{proof}
\begin{remark}
\rm
In fact, we do not have to assume the existence of a subsolution to \rf{P},
because $\uu = 0$ is a subsolution.
\end{remark}

\subsection{Existence of the finite pull-in voltage}

\begin{proof}[Proof of Theorem \ref{t-pulin}]
{\it Step 1. Integral solutions to \rf{P} exist for small values of $\la>0$.}
Define 
\aa{
\bar u(r) = \int_r^1 \ffi^{-1} \left(t, \la_0 F(t)\right) dt,
}
where $F(\fdot)$  is defined in \ref{H4}.
Note that the right-hand side is continuous in $\la_0\in [0,1]$ 
by the dominated convergence theorem and the continuity of $\ffi^{-1}$
in the second argument. By \ref{H2} and since $\ffi^{-1}(t,0) = 0$, 
we can choose $\la_0>0$ small enough, 
such that $\rho_{\la_0} = \sup_{r\in [0,1]} \bar u(r) = 
\int_0^1 \ffi^{-1} \left(t, \la_0 F(t)\right) dt < 1$. 
This implies that
$\int_0^1 \frac{f(s)}{g(\bar u(s))} ds \lt
g(\rho_{\la_0})^{-1} F(1) <\infty$.
Furthermore, since $\frac{g(\rho_{\la_0})}{g(\bar u(r))}\lt 1$,
for $\la\in (0,\la_0 \, g(\rho_{\la_0})]$, we obtain
\mm{
\bar u(r) = \int_r^1 \ffi^{-1} \left(t, \la_0 \, F(t)\right) dt \gt
 \int_r^1 \ffi^{-1} \l(t, \la_0 \, g(\rho_{\la_0}) \int_0^t \frac{f(s)}{g(\bar u(s))} \, ds\r) dt \\
 \gt  \int_r^1 \ffi^{-1} \l(t, \la \int_0^t \frac{f(s)}{g(\bar u(s))} \, ds\r) dt. 
 }
Therefore, $\bar u$ is a supersolution for \rf{P}.
Since $0$ is a subsolution for \rf{P}, by Proposition \ref{tss},
for each $\la \in (0, \la_0 \, g(\rho_{\la_0})]$, there exists an integral solution to \rf{P}.

{\it Step 2. Upper bound \rf{pulin}.} 
For $v\gt 0$, define $\td P_{d,e}(v) = \max\{v^d,v^e\}$ and note that
$\td P_{d,e} = p_{d,e}^{-1}$.
The invertibility of $\ffi(r,\fdot)$, assumption \ref{H3}, and the
fact that $\mc P(1) = \sum_i c_i$ imply that for all $r\in (0,1)$,
\aa{
\ffi(r,v) \lt \mc P(1) a(r) \td P_{d,e}(v) \quad \text{and} \quad \ffi^{-1}(r,v)
 \gt p_{d,e}\l(\frac{v}{\mc P(1) a(r)}\r)\gt \frac{p_{d,e}(v)}{P_{d,e}(\mc P(1) a(r))}.
}
Let $u(r)$ be an integral solution to \rf{P}. Then, for all $r\in (0,1)$,
\mmm{
\lb{comp1111}
u(r) =  \int_r^1 \ffi^{-1}\left(t,\la \int_0^t \frac{f(s)}{g(u(s))}\right) dt 
\gt   \int_r^1 \mc A^{-1} p_{d,e}(\la)\, p_{d,e}\!\left(\int_0^t \frac{f(s)}{g(u(s))} ds\right) dt\\
\gt \mc A^{-1}  p_{d,e}(\la)  p_{d,e}\!\left(\int_0^r \frac{f(s)}{g(u(s))} ds\right) (1-r)
\gt  \mc A^{-1}  \, p_{d,e}(\la) \, \frac{p_{d,e} \left(F(r)\right)}
{P_{d,e}(g(u(r))} \, (1-r),
}
where $\mc A = P_{d,e}\big(\mc P(1)\sup_{r\in [0,1]}a(r)\big)$.
By \rf{comp1111} and the monotonicity of $g$,
\aa{
  \mc A P_{d,e}(g(0)) \gt
 \mc A \sup_{u\in [0,1]} \{ u P_{d,e}(g(u))\} \gt  \mc A\, u(r) P_{d,e}(g(u(r))
\gt p_{d,e}(\la) \, p_{d,e} \!\l(F(r)\r)  (1-r).
}
 Evaluating the right-hand side at 
 $r=\frac12$, we obtain that $p_{d,e}(\la)$ 
 is bounded from above. Therefore, the pull-in voltage $\la^*$ is finite and
  \aa{
 p_{d,e}(\la^*) 
 \lt 2 \mc A  \,      P_{d,e}(g(0))
 \big[p_{d,e} \l(F(\sfrac12)\r) \big]^{-1}.
 }
The above inequality implies \rf{pulin}.
 
 {\it Step 3. Lower bound \rf{l-pulin}.}
 To obtain the lower bound, consider the function $\Phi(\la)
 =\int_0^1\ffi^{-1}(s,\la\, F(s)) ds$ for $\la\in [0,1]$, and note that 
 $\Phi(1)<+\infty$ by \ref{H2}. 
 
 Suppose first that $\Phi(1) < 1$. Then, in {\it Step 1}, we can take $\la_0=1$. Also, 
 note that $\rho_{\la_0}:= \Phi(\la_0) = \Phi(1) < 1$. By {\it Step 1}, the integral solution
 to \rf{P} exists for all $\la \in [0, g(\Phi(1))]$. Therefore, $\la^*\gt g(\Phi(1))$.
 
 Now suppose that $\Phi(1)\gt 1$. Since $\Phi(0) = 0$ and the function
 $\la\to\Phi(\la)$ is strictly increasing, find $\la_0$ such that
 $\Phi(\la_0) = \dl$, where  $\dl\in (0,1)$.
 Note that $\la_0 = \Phi^{-1}(\dl)$ and $\rho_{\la_0}= \Phi(\la_0) = \dl$.
By {\it Step 1}, the integral solution
 to \rf{P} exists for all $\la \in (0, \Phi^{-1}(\dl)g(\dl)]$. 
 This implies that $\la^*\gt \sup_{\dl\in [0,1]}\Phi^{-1}(\dl)g(\dl)$. 
\end{proof}

\section{Touchdown solution}
\lb{s3}

\subsection{Maximal and minimal branches of solutions}

\begin{proof}[Proof of Theorem \ref{prop1111}.]
We prove the existence of a maximal branch of solution. The existence of a minimal
branch is obtained in the same way.

{\it  Step 1. The set of integral solutions to \rf{P} for a given $\la$ is partially ordered.}
We denote the above-mentioned set by $\La_\la$. 
By Theorem \ref{t-pulin}, for each $\la\in (0,\la^*)$, there exists
an integral solution to \rf{P}, so $\La_\la$ is non-empty. 

Below, we assume that $\la\in (0,\la^*)$.
The set $\La_\la$ is partially ordered. Indeed, we can compare two solutions
$u_\la$ and $v_\la$ from $\La_\la$ if $u_\la(r)\lt v_\la(r)$ 
for all $r\in [0,1]$ or vice versa. On the other hand, two solutions
$u_\la$ and $v_\la$ from $\La_\la$ cannot be compared if there are subintervals of $(0,1)$
where $u_\la>v_\la$ and where $v_\la>u_\la$.

{\it Step 2.  Each chain of $\La_\la$ has a maximal element.}
If $\La_\la$ is finite, the statement is obvious. Below, we consider the case
when $\La_\la$  is infinite. Note that
the set of comparable solutions corresponding to the same $\la$ forms a chain,
which we denote by $C_\la$. We show that $C_\la$ has a maximal element.
For any two elements $u,v\in C_\la$, $u\lt v$, for all $r\in (0,1]$ we have
\mmm{
\lb{c1111}
0\lt v(r) - u(r) = v(0) - u(0) \\
- \int_0^r \l [ \ffi^{-1}\l (t, \la \int_0^t \frac{f(s)}{g(v(s))} ds \r )
- \ffi^{-1}\l (t, \la \int_0^t \frac{f(s)}{g(u(s))} ds \r ) \r] dt 
\lt v(0) - u(0). 
}
Define the number
\aa{
a_{\max} = \sup_{u\in C_\la} \{u(0)\}.
}
Let $u_n\in C_\la$ be a non-decreasing sequence such that 
$a_{\max} = \lim_{n\to\infty} u_n(0)$.
Note that $\{u_n\}$ is a Cauchy sequence in $\C[0,1]$. Indeed, since $\{u_n(0)\}$
is a Cauchy sequence, for $m>n$, we have 
$0\lt u_n(r) - u_m(r) \lt u_n(0) - u_m(0) \to 0$ as  $m,n \to +\infty$.
Therefore, there exists a function $v\in \C[0,1]$, $v=\lim_{n\to \infty} u_n$.
The dominated convergence theorem  implies that
$v$ is a solution to equation  \rf{eq1111}. 
Let us  first show that $v\in C_\la$, i.e., $v$ can be compared with any other element
from $C_\la$. Let $u\in C_\la$ be arbitrary.

Case 1. $v(0) < u(0)$. Then, for all $n\in \Nnu$, $u_n(0) < u(0)$.
Since $u_n$ and $u$ are comparable, we have that $u_n(r)\lt u(r)$ for all 
$r\in[0,1]$. This implies that $v(r)\lt u(r)$ for all $r\in [0,1]$. 

Case 2. $v(0) > u(0)$. Then, there exists $N>0$ such that $u_n(0) > u(0)$
for all $n\gt N$. Since $u_n$ and $u$ are comparable, we have that $u_n(r)\gt u(r)$
for all  $r\in [0,1]$. This implies that $v(r)\gt u(r)$ for all $r\in [0,1]$.

Case 3. $v(0) = u(0)$. Suppose $[0,r_0]$ be the maximal interval where 
$v(r) = u(r)$ and $r_0<1$.
In a small right neighborhood of $r_0$, denote it by $(r_0,r_0+\eps)$, 
we either have $v(r) > u(r)$ or $v(r)<u(r)$. However, this contradicts to the 
following equation which has to be fulfilled for $r\in (r_0,r_0+\eps)$:
\aa{
v(r) - u(r) = 
- \int_{r_0}^r \l [ \ffi^{-1}\l (t, \la \int_0^t \frac{f(s)}{g(v(s))} ds \r )
- \ffi^{-1}\l (t, \la \int_0^t \frac{f(s)}{g(u(s))} ds \r ) \r] dt 
}
Hence, in the Case 3, we have that $r_0=1$ and $u = v$.

Therefore, $v$ can be compared with other elements from $C_\la$.
Let us show that $v$ is a maximal element of $C_\la$.
Suppose there is another element $w\in C_\la$ such that $w\gt v$.
However, one must have $w(0) = v(0)$. By \rf{c1111}, $w(r)=v(r)$ for 
all $r\in [0,1]$.

{\it Step 3. $\La_\la$ contains a unique maximal element.}
By the Zorn lemma, the set $\La_\la$ contains a maximal element $u_{\max}(r)$.
Suppose there is another maximal element $v\in \La_\la$. Since $u_{\max}\in \La_\la$,
then $v\lt u_{\max}$. By the same reasoning, $u_{\max}\lt v$. This proves the statement
of Step 3. 

Maximal solutions, which we proved to exist for each $\la\in(0,\la^*)$, form a maximal branch.
\end{proof}
\begin{proposition}
Let $u_\la$, $\la\in (0,\la^*)$, be the maximal branch for \rf{P}. Then, for each pair
$\la_1,\la_2 \in (0,\la^*)$, $\la_1 < \la_2$, it holds that $u_{\la_1}\lt u_{\la_2}$.
\end{proposition}
\begin{proof}
Below, for simplicity of notation, we introduce
\aa{
\chi(s,u) = \frac{f(s)}{g(u)}.
}
Take $\la_3\in(\la_2,\la^*)$. We have that $u_{\la_1}$ is a subsolution to ($P_{\la_2}$) and $u_{\la_3}$ is a supersolution for the same problem.
Indeed,
\aa{
u_{\la_1} = \int_r^1 \ffi^{-1} \l (t, \la_1 \int_0^t \chi(s,u_{\la_1}(s)) ds\r) dt
< \int_r^1 \ffi^{-1} \l (t, \la_2 \int_0^t \chi(s,u_{\la_1}(s)) ds\r) dt.
}
In the same way,
\aa{
u_{\la_3} 
> \int_r^1 \ffi^{-1} \l (t, \la_2 \int_0^t \chi(s,u_{\la_3}(s)) ds\r) dt.
}
By Proposition \ref{tss}, there exists an integral solution $v_{\la_2}$ to   ($P_{\la_2}$) 
such that $u_{\la_1}\lt v_{\la_2}$. This implies that $u_{\la_1}\lt u_{\la_2}$
by the maximality of $u_{\la_2}$.
\end{proof}
By the same argument, we have the following corollary.
\begin{corollary}
\lb{cor4444}
Let $u_\la$, $\la\in (0,\la^*)$ be the minimal branch for \rf{P}. Then, for each pair
$\la_1,\la_2 \in (0,\la^*)$, $\la_1 < \la_2$, it holds that $u_{\la_1}\lt u_{\la_2}$.
\end{corollary}
\begin{corollary}
\lb{cor2222}
Let $u_\la$ be the maximal or minimal branch. Then, as $\la\to\la^*$,
$u_\la$ converges pointwise to a finite function.
\end{corollary}
\begin{proof}
For each $r\in [0,1]$, $u_\la(r)$ is bounded and increasing in $\la$. This implies the statetement
of the corollary. 
\end{proof}

\subsection{Existence, non-existence, and uniqueness of a touchdown solution}

In what follows, we will need the following lemma.
\begin{lemma}
\lb{lem1111}
For any $\sg\in (0,1)$, the family of integral solutions $u_\la$, $\la\in [\la_0,\la^*)$, 
$\la_0\in (0,\la^*)$, to problem \rf{P} is uniformly 
bounded by a number $1-\dl$  (where $\dl\in (0,1)$ depends
on $\sg$) on the subinterval $[\sg,1]\sub (0,1]$.
\end{lemma}
\begin{proof}
For simplicity of notation, we write $u$ instead of $u_\la$ for the solution to \rf{P}.
Let $\ro_\dl = 1-\dl$ and let $\mu_\dl\in (0,1)$ be such that
$u(\mu_\dl) = \ro_\dl$ (if exists). Note that if $u(0)>1-\dl$, then $\mu_\dl$ always exists
and is uniquely defined (since $u'(r)<0$). Otherwise, if $u(0)\lt 1-\dl$, then
$u(r) \lt 1-\dl$ for all $r\in [0,1]$. Below, we consider the case $u(0)>1-\dl$.

Evaluating the  both sides of the equation in \rf{P} at 
$\mu_\dl t$, $t\in (0,1)$ and integrating with respect to $t$ from $r$ to $1$, 
where $r\in [0,1]$, we obtain
\aa{
u(\mu_\dl r) - \ro_\dl = \mu_\dl \int_r^1 \ffi^{-1}\l(\mu_\dl t, \la \int_0^{\mu_\dl t} \frac{f(s)}
{g(u(s))} ds\r) dt.
}
Now the same computation as \rf{comp1111}, for all $r\in (0,1)$, implies
\aa{
u(\mu_\dl r) - \ro_\dl    \gt  
 \mc A^{-1}  \, p_{d,e}(\la) \, 
\frac{ \mu_\dl \, p_{d,e}\left( F(\mu_\dl r)\right)}{P_{d,e}(g(u(\mu_\dl r))}  \, (1-r),
}
where $\mc A$ is the same as in  \rf{comp1111}.
From here, we obtain
\aa{
K\!\!\sup_{u\in [\ro_\dl,1]} \!\!\big\{(u - \ro_\dl) P_{d,e}(g(u))\big\}
\gt 
K\l(u(\mu_\dl r) - \ro_\dl\r) \! P_{d,e}
\!\l(g(u(\mu_\dl r)\r)
\gt \mu_\dl \, p_{d,e}\! \l(F(\mu_\dl r)\r)\!(1-r),
}
where $K = \frac{\mc A}{p_{d,e}(\la_0)}$. 
As in the case of \rf{comp1111}, the above inequality holds for 
$r\in (0,1)$. Evaluating the right-hand side at $r=\frac12$,
we obtain
\aa{
Q(\mu_\dl):=\mu_\dl \, p_{d,e} \Big(F\!\l(\frac{\mu_\dl}2\r)\Big) \lt 
2K \sup_{u\in [1-\dl,1]} (u - 1+\dl) P_{d,e}(g(u))
\lt 2K\, \dl \, P_{d,e}\l(g(1-\dl)\r).
}
Note that $p_{d,e}(\fdot)$ is strictly increasing. 
Next, since $F(r)>0$ for all $r\in (0,1)$ by \ref{H4}, $f(r)$ cannot be an identical 
zero in a small right neighborhood $(0,\eps)$ of  $0$.
Therefore, $F$ is strictly increasing on $(0,\eps)$. 
Therefore, $Q(\fdot)$ is strictly increasing on $(0,2\eps)$. 
Hence,
$\mu_\dl\to 0$ as $\dl\to 0$ uniformly in $u$, an integral solution to 
\rf{P}, and $\la \in [\la_0,\la^*]$. Fix $\sg\in (0,1)$ and 
choose $\dl>0$  such 
that $Q^{-1} \l(2K\, \dl \, P_{d,e}\l(g(1-\dl)\r)\r) < \sg$. 
Then, on $[\sg,1]$, integral solutions to \rf{P} are bounded by $1-\dl$
uniformly in $\la\in [\la_0,\la^*]$.
\end{proof}

\begin{proof}[Proof of Theorem \ref{th2222}]
By Lemma \ref{lem1111}, for any $N\in \Nnu$, the family $u_\la$, $\la \in [\la_0,\la^*)$
($\la_0>0$), is bounded on $[\frac1N, 1]$ by a constant $1-\dl$, where $\dl$ depends on
$N$. Therefore,  by the bounded convergence theorem,
for all $r\in [0,1)$ and $N\in \Nnu$,
\aa{
\int_r^1\!\ffi^{-1}\!\left(t, \la^* \int_{\frac1N}^t \frac{f(s)}{g(u_{\la^*}(s))} ds\right)\! dt
=\lim_{\la\to \la^*} \int_r^1\!\ffi^{-1}\left(t, \la \int_{\frac1N}^t \frac{f(s)}{g(u_\la(s))} ds\right)\! dt
\lt \lim_{\la\to \la^*} u_\la  \lt 1.
}
By the monotone convergence theorem, the sequence
\aa{
\l \{\ffi^{-1}\left(t, \la^* \int_{\frac1N}^t \frac{f(s)}{g(u_{\la^*}(s))} ds\right) \r \}_{N=1}^\infty
}
converges to an integrable function as $N\to \infty$. Therefore,
the limit function 
\aa{
\ffi^{-1}\left(t, \la^* \int_0^t \frac{f(s)}{g(u_{\la^*}(s))} ds\right)
}
is integrable on $[r,1)$, for any $r\in [0,1)$, and one can pass to the limit as $\la\to\la^*$ 
in equation \rf{eq1111} by the dominated convergence theorem.  Therefore,
$u_{\la^*}$ is an integral touchdown solution to problem \rf{P}.
\end{proof}
The proof of Theorem \ref{unique2222} on the uniqueness of a touchdown
solution, whenever it exists, is a direct consequence of the following proposition.
\begin{proposition}
\lb{unique1111}
Assume \ref{H1}--\ref{H3}. Let $u_{\la^\#}$, $\la^\#\in (0,\la^*]$, 
be an integral touchdown solution to  $(P_{\la^\#})$,
i.e., $u_{\la^\#}(0) = 1$. Then, $\la^\#$ is the pull-in voltage, i.e.,
 $\la^\#=\la^*$. Moreover, $u_{\la^\#}=u_{\la^*}$
 and the solution $u_{\la^\#}$ is maximal.
  
\end{proposition}
\begin{proof}
Suppose $\la^\#<\la^*$.
By Theorem
\ref{prop1111}, there exists a (unique) maximal branch of integral solutions $v_\la$
to \rf{P} when $\la$ varies over $[\la^\#, \la^*]$.
By Corollary \ref{cor4444},
\aaa{
\lb{max2222}
u_{\la^\#} \lt  v_{\la^\#} \lt v_\la, \quad \text{for all}\;\;
\la \in [\la^\#, \la^*].
}
However, since $u_{\la^\#}(0) = 1$, then $v_\la(0) = 1$ for all $\la \in  [\la^\#, \la^*]$.
By Theorem \ref{th2222}, the pointwise limit $\lim_{\la\to \la^*} v_\la(r)=u_{\la^*}(r)$
exists and is an integral touchdown solution to $(P_{\la^*})$.  
Clearly, $u_{\la^*}(0) = u_{\la^\#}(0) = 1$ and $u_{\la^\#}(r) \lt u_{\la^*}(r)$ by
\rf{max2222}.
Let $[0,r_0]\sub [0,1]$ be the maximal interval where $u_{\la^\#}(r) =  u_{\la^*}(r)$.
If $r_0 < 1$, then
in a small right neighborhood of $r_0$, denote it by $(r_0,r_0+\eps)$, 
we have that $u_{\la^\#}(r) < u_{\la^*}(r)$. However, this contradicts to the 
following equation which has to be fulfilled for all $r\in (r_0,r_0+\eps)$:
\aa{
u_{\la^\#}(r) - u_{\la^*}(r) = 
- \!\int_{r_0}^r \l [ \ffi^{-1}\!\l(t, \la^\# \!\int_0^t \frac{f(s)}{g(u_{\la^\#}(s))} ds \r )
- \ffi^{-1}\!\l (t, \la^*\! \int_0^t \frac{f(s)}{g(u_{\la^*}(s))} ds \r ) \r] dt. 
}
Therefore, $r_0=1$, and hence, $u_{\la^\#} =  u_{\la^*}$. This implies that $\la^* = \la^\#$
since the above equation should be fulfilled for all $0\lt r_0<r\lt 1$.
Furthermore, $v_{\la^\#} = u_{\la^\#}$ by exactly the same argument. Hence, 
$u_{\la^\#}$ is maximal.
\end{proof}
\begin{proposition}
\lb{prop2222}
Assume \ref{H1}--\ref{H3} and let $u_\la$ be as in Theorem \ref{th2222}.
Further assume that $\lim_{\la\to\la^*} u_\la(0) < 1$. Then, the pointwise
 limit $u_{\la^*}(r) = \lim_{\la\to\la^*}u_\la(r)$ is an integral solution to \rf{P}
 which is not a touchdown solution. 
\end{proposition}
\begin{proof}
Since in this case $g(u_\la(r))$ is bouded by $g(u_{\la^*}(0))$, where 
$u_{\la^*}(0)<1$, we can pass to the limit  as $\la\to\la^*$ in equation \rf{eq1111}
by the bounded convergence theorem.  
\end{proof}
\begin{corollary}[Corollary of Theorem \ref{th2222} and Proposition \ref{prop2222}]
\lb{cor-pulin}
Assume \ref{H1}--\ref{H3}. Then,
there exists a positive number $\la^*>0$ such that
for all $\la\in (0,\la^*]$, there is an integral solution to \rf{P} and
for all $\la>\la^*$, there is no integral solution to \rf{P}.
\end{corollary}

\section{$\C^2$-regularity of integral solutions}
\lb{s4}

\begin{theorem}
Assume \ref{H1}--\ref{H3}. Furthermore, we assume that
 $\ffi\in \C^1([0,1]\x \Rnu)$. Let $u(r)$ be an integral solution to \rf{P}.
 Then, $u\in \C[0,1]\cap \C^2(0,1)$. In particular,  the touchdown solution 
 to \rf{P} is $\C[0,1]\cap \C^2(0,1)$-regular.
\end{theorem}
\begin{proof}
The equation in \rf{P} is equivalent to
\aa{
-u'(r) = \ffi^{-1}\l(r, \la\int_0^r\frac{f(s)}{g(u(s))}\r).
}
By the inverse function theorem, the map $\ffi^{-1}(r,v)$ is differentiable
in the second argument.  Thus, in order to show that $u\in \C^2(0,1)$,
we have to show that $\ffi^{-1}(r,v)$ is differentiable in $r$.
For each  fixed $v\in\Rnu$, the equation $\Psi(r,y): = \ffi(r,y) - v = 0$ defines 
$y$ as an implicit function of $r$. By the theorem on the differentiability
of an implicit function, we have that $\frac{dy}{dr} = -\frac{\pl_r \Psi(r,y)}{\pl_y \Psi(r,y)}$.
However, the implicit function $y(r) = \ffi^{-1}(r,v)$ is defined uniquely. Consequently,
$\ffi^{-1}(r,v)$ is differentiable in $r$ and 
$\pl_r \ffi^{-1}(r,v) =-\frac{\pl_r \ffi(r,y)}{\pl_y \ffi(r,y)}\big|_{y=\ffi^{-1}(r,v)}$.
\end{proof}

\section{Touchdown solution to \rf{G}: existence and asymptotics}
\lb{s5}
This section is dedicated to model \ref{E1}. Here we prove Theorem \ref{ss2222}.
\subsection{Equivalent problem in a neighborhood of infinity}
As we mentioned before, the equation in \rf{P} 
can be written in the form \rf{G} which, in turn, implies the following equation
 obtained by differentiation:
\aaa{
\lb{G1}
L(\al,\beta,\gm)[u] = \la \frac{h(r)}{g(u(r))}.
}
Note that its solution $u$ is an integral solution if and only if
\aaa{
\lb{integral1111}
\lim_{r\to 0} r^\al |u'(r)|^\beta u'(r) =  0.
}
It is straighforward to verify that for $r\in (0,1)$, 
equation \rf{G1} can be rewritten as 
 \aa{
u''(r)|u'(r)|^\beta + \al \,\frac{u'(r)|u'(r)|^\beta}{r} + \frac{\la\, r^{\gm-\al}\, 
h(r)}{g(u(r))} = 0.
}
By the change of variable $r = c\,e^{-t}$, where $c=\la^{-\frac1{\te}}$,
 we transform \rf{G1} to the problem
\aaa{
\lb{w-dif1111}
(\beta+1)|v'|^\beta v'' - (\al -\beta-1) |v'|^{\beta}v' + \frac{e^{-\te t} h(t)}{g(v)} = 0, 
\quad t\in (\ln c,+\infty),
}
where $\te = \gm +2+ \beta - \al$ and, with a slight abuse of notation, we re-used the symbol
$h(t)$ for the function $h(c\,e^{-t})$.
The boundary condition $u(1) = 0$
becomes  $v(\ln c) = 0$. Note that \rf{w-dif1111}
 is equivalent to \rf{G1}, but given in a neighborhood of infinity.

Since we are interested in integral solutions, condition \rf{integral1111} implies
that
\aaa{
\lb{w-int1111}
|v'(t)|^{\beta}v'(t) =  \int_t^{+\infty} \!\!e^{(\al-\beta-1)(t-s)}\, \frac{e^{-\te s} h(s)}{g(v(s))} \,ds
}
and it is equivalent to the convergence of the integral on 
 the right-hand side of \rf{w-int1111}.

While we call a solution $u(r)$ to \rf{G} with $u(0) = 1$ a touchdown solution, 
the respective solution $v(t) = u(c\, e^{-t})$ to \rf{w-int1111}, such that $\lim_{t\to+\infty} v(t) = 1$,  
will be referred to below as a {\it singular solution}.  
\subsection{Construction of a singular solution to  equation \rf{w-int1111}}
The following below Schauder's fixed point theorem
can be found in \cite{smart}, Section 4.
\begin{theorem}(Schauder's fixed point theorem)
\lb{schauder}
 Let $X$ be a normed space and let $\mathscr T$ be a compact continuous map
$X \to X$. Then, $\mathscr T$ has a fixed point.
\end{theorem}
\begin{theorem} 
\lb{t1111}
Assume \ref{A1}--\ref{A4}. 
Then, there exists $T>0$ such that on $[T,+\infty)$, equation \rf{w-int1111}
possesses the singular solution 
\aaa{
\lb{v1111}
v^*(t) = g^{-1}(\kp\, e^{-\te \cp t} + x(t)), \qquad 
x(t) = o(e^{-\te \cp t}),
 }
 where the constants $\kp>0$ and $\cp\in (0,1)$ are given by \rf{co1111}.
\end{theorem}
\begin{proof}
The idea of the proof is to substitute $v^*$, defined by \rf{v1111}, into 
equation \rf{w-dif1111} and to set up a fixed point argument, with respect to $x$, 
in a certain normed space by means of Theorem \ref{schauder}.

{\it Step 1. The set of equations defining the fixed point map.} 
Thus, we search for a singular solution in the form \rf{v1111}, where
the choice of $\cp \in (0,1)$ and $\kp>0$ will be explained later.
The last term in \rf{w-dif1111} transforms via the substitution 
 \rf{v1111}  as follows:
\aaa{
\lb{sig2222}
& \frac{e^{-\te t} h(t)}{\kp \, e^{-\te \cp t} + x(t)}  =
\kp^{-1}e^{-\te(1-\cp) t} h(t)  \Big(1- \frac{\X(t)}{1+\X(t)}\Big),\\
&\text{where} \;\; \X(t) = x(t) e^{\te \cp t} \kp^{-1}.\lb{x1111}
}
Note that if $T>0$ is large enough, then $|\X(t)|<\frac14$ for all 
$t\gt T$. 
Hence, by \ref{A4}, if the solution in the form \rf{v1111} exists for a large enough $T>0$, 
it is an integral solution.
In particular, ${v^*}'(t)>0$, so $v^*(t)$ increases in $t$.

For the fixed point argument, we will use the following set of equations.
First, by using \rf{w-int1111} and \rf{sig2222}, we represent ${v^*}'(t)^{\beta+1}$ as 
\aaa{
\lb{psiy1111}
{v^*}'(t)^{\beta+1}  = \psi(t) + y(t),
}
where 
\eqn{
\lb{psi1111}
&\psi(t) = \kp^{-1} \int_t^{+\infty} \!\!e^{(\al-\beta-1)(t-s)}\, e^{-\te(1-\cp) s} h(s) ds,\\
&y(t) =  - \kp^{-1} \int_t^{+\infty} \!\!e^{(\al-\beta-1)(t-s)}\, e^{-\te(1-\cp) s} h(s) 
\frac{\X(s)}{1+\X(s)} ds.
}
By boundedness of $h$, the integrals \rf{psi1111} converge for sufficiently large $T>0$.
Derivating \rf{v1111} with respect to $t$, gives
\aaa{
\lb{psi+y1111}
\l[ \frac{- \te\cp \kp\, e^{-\te \cp t} + x'}{g'\!\l(g^{-1}(\kp \, e^{-\te \cp t} + x)\r)}\r]^{\beta+1}
= \psi(t) + y(t).
}

{\it Step 2. The fixed point map.}
Introduce the normed space
\eqn{
&X_T = \l\{x\in \C[T,+\infty) \;\; \text{such that} \;\; x(t)  = 
o(e^{-\te \cp t})\;\; 
\text{as} \;\; t\to+\infty    \r\}\\
&\text{with the norm} \quad \|x\|_\T: =  \sup_{t\gt T} |x(t)|.\notag
}
Furthermore, introduce the subset of $X_T$:
\aa{
&\Xi_T  = 
\Big\{x\in X_T: 
\|x\|_t \lt \frac{\kp}4 \, e^{-\te \cp t}\; \; \forall \, t\gt T\Big\} =
 \Big\{x\in X_T: \|\X\|_T \lt \frac14\Big\}, \\
&\|x\|_t = \sup_{s\gt t} |x(s)|,
}
where $\X$ is defined via $x\in X_T$ by \rf{x1111}.
Note that for all $x\in \Xi_T$, 
 the argument of $g^{-1}$ in \rf{v1111} is positive, and hence,
$v^*$ is well-defined. 

First, we will define the map $\Gm: \Xi_T  \to X_T$ while fixing
the parameters $\cp$ and $\kp$ in such a way that
$\Gm(x)\in X_T$ for $x\in \Xi_T $.
We then extend the map $\Gm$ from $\Xi_T $ to $X_T$, obtaining 
by this a map $\hat\Gm: X_T\to X_T$.
We will then apply Schauder's fixed point theorem 
to show that the map $\hat \Gm: X_T\to X_T$ has a fixed point $x^0$. By choosing 
$T>0$ sifficiently large, we achieve that $x^0$ is a fixed point of $\Gm$.

More specifically, we define $\Gm(x)$ by
expressing $x'$ from the left-hand side of \rf{psi+y1111}, 
and integrating the resulting equation from $t$ to $+\infty$. Namely, given $x\in \Xi_T $,
\aaa{
\lb{gm1111}
\Gm(x)(t) = - \int_t^{+\infty}\l((\psi(s)+y(s))^{\frac1{\beta+1}}
g'\l(g^{-1}(\kp \, e^{-\te \cp s} + x(s))\r) + \te\cp \kp\, e^{-\te \cp s}\r) ds,
}
where $y$ is defined, also via $x\in \Xi_T $, by the second equation 
in \rf{psi1111}. The function $\psi$, explicitly given by the first equation in
\rf{psi1111}, can be regarded as known. Moreover, since $h(t) = C(1+o(1))$ 
as $t\to+\infty$, the first expression in \rf{psi1111} implies that
\aaa{
\lb{psi2222}
\psi(t) = \frac{C}{\kp (\gm+1 - \te\cp)} \,e^{-\te(1-\cp)t}(1+o(1)),
} 
where we have taken into account that $\te(1-\cp) + (\al-\beta-1) = \gm+1-\te\cp$.
Since for $x\in \Xi_T $, $|\X(t)|\lt \frac14$ for all $t\gt T$, it holds that
\aaa{
\lb{y1111}
|y(t)| \lt \frac43\,\|\X\|_t \,\psi(t) < \frac{\psi(t)}{3},
}
Therefore, $(\psi(t)+y(t))^{\frac1{\beta+1}}$ in \rf{gm1111} is well-defined and
\aa{
(\psi(t)+y(t))^{\frac1{\beta+1}} 
= \psi(t)^{\frac1{\beta+1}}\big(1+o(1)\big).
}
Next, by Remark \ref{a2222}, which is equivalent to  \ref{A2},
we have
\aaa{
\lb{2nd1111}
&g'\l(g^{-1}(\kp \, e^{-\te \cp s} + x(s))\r) = -A(\kp e^{-\te \cp  s}+x)^q + o(e^{-\te \cp  q s})\\
&= -A\,\kp^q e^{-\te \cp q s}(1+o(1))^q + o(e^{-\te \cp q s})
= -A\,\kp^q e^{-\te \cp q s} (1+o(1)).\notag
}
By the above formula, \rf{gm1111} and \rf{psi2222}, we obtain
\mmm{
\lb{gm2222}
\Gm(x)(t) = \int_t^{+\infty}  \Big( A \kp^q \l[\frac{C}{\kp (\gm+1 - \te\cp)}\r]^{\frac1{1+\beta}}
\exp\l\{-\te \cp q s - \frac{\te(1-\cp)s}{\beta+1}\r\} (1+o(1))\\
- \te\cp \kp\, \exp\{-\te \cp s\} \Big) ds.
}
Thus, by setting  
\aaa{
\lb{eqs1111}
\te \cp q + \frac{\te(1-\cp)}{\beta+1} = \te \cp \quad \text{and} \quad
A \kp^q \l[\frac{C}{\kp (\gm+1 - \te\cp)}\r]^{\frac1{1+\beta}} = \te\cp \kp,
}
we obtain that $\Gm(x)(t) = o(e^{-\te \cp t})$. Therefore, $\Gm(x)\in X_T$
if $x\in\Xi_T $.
Furthermore, the first equality in \rf{eqs1111} implies the expression
for $\cp$ in \rf{co1111}, while the second equality in \rf{eqs1111}
 implies the expression for $\kp$.
 
 Thus, we have constructed the map $\Gm:\Xi_T \to X_T$.  
 Let for 
 $x\in X_T$, $\X_\f(t) = \min\{\X(t),\frac14\}$ if
 $\X(t)\gt 0$ and $\X_\f(t) = \max\{\X(t),-\frac14\}$ if
 $\X(t)< 0$,
 where (we recall)
 $\X$ is defined via $x\in X_T$ by \rf{x1111}. Further,
 if $x\in X_T$, we define $x_\f(t) = \kp\, e^{-\te \cp t} \X_\f(t)$.
 Clearly, $x_\f\in \Xi_T$  if $x\in X_T$.  Thus, we define 
 \aa{
\hat \Gm: X_T\to X_T, \quad \hat \Gm(x) = \Gm(x_\f).
 }

{\it Step 3. Continuity of $\hat \Gm$.}
Take $x^1,x^2\in\Xi_T $. By \rf{gm1111} and \rf{y1111},
\mm{
|\Gm(x^1)(t) - \Gm(x^2)(t)|
\lt \int_t^{+\infty} \Big(
 |g'(g^{-1}(\kp e^{-\cp \te s} + x^1))|\, \big|(\psi+y^1)^\frac1{\beta+1}
- (\psi+y^2)^\frac1{\beta+1}\big|\\
+ \Big(\frac43\Big)^\frac1{\beta+1} \psi(s)^\frac1{\beta+1} 
\big|g'(g^{-1}(\kp e^{-\cp \te s} + x^1)) - g'(g^{-1}(\kp e^{-\cp \te s} + x^2))\big|
\Big) ds,
}
where $y^1$ and $y^2$ are
functions associated with $x^1$ and $x^2$, respectively,
by formula \rf{psi1111}. By Remark \ref{a2222},
which is equivalent to \ref{A2}, $|g'(g^{-1}(\kp e^{-\cp \te t} + x^1(t)))| \lt  K_1 e^{-q \cp \te t} $,  
where $K_1>0$ is a constant.
Further, by \rf{psi1111},
\aa{
|y^1(t)- y^2(t)| &\lt 
\kp^{-1}\!\! \int_t^{+\infty} \!\!\!e^{(\al-\beta-1)(t-s) -\te(1-\cp) s} 
\frac{ h(s) |\X^1 - \X^2|}{(1+\X^1)(1+\X^2)}ds\\	
&\lt \frac{16}9\, \kp^{-1} \psi(t)\, e^{\te \cp t} \|x^1-x^2\|_t,	
}
and consequently, there exists a constant $K_2>0$ such that
\aa{
\big|(\psi+y^1)^\frac1{\beta+1} - (\psi+y^2)^\frac1{\beta+1}\big|
\lt 
K_2\, \psi(t)^\frac1{\beta+1} e^{\te \cp t} \|x^1-x^2\|_t.
}
Furthermore, 
\aaa{
\lb{rhsg}
\big|g'(g^{-1}(\kp e^{-\cp \te t} \!+ x^1)) - g'(g^{-1}(\kp e^{-\cp \te t} \!+ x^2))\big|\!
\lt\! \sup_{x\in\Xi_T } \l|\frac{g''(g^{-1}(\kp e^{-\cp \te t} \!+x))}{g'(g^{-1}(\kp e^{-\cp \te t} \!+x))}\r|
|x^1(t)-x^2(t)|.
}
The supremum on the right-hand side, can be estimated by \ref{A2}
and L'Hopital's rule:
\aa{
-A = \lim_{u\to 1_-}\frac{g'(u)}{g(u)^q} = \lim_{u\to 1_-} \frac{g''(u)}{q g(u)^{q-1} g'(u)}.
}
From here, we obtain that as $u\to 1_-$,
\aa{
\frac{g''(u)}{g'(u)} = -A\,q\, g(u)^{q-1}(1+o(1)).
}
Therefore, the left-hand side of \rf{rhsg} can be estimated from above by the term
\aa{
K_3\, e^{\cp\te(1-q)t}\|x^1- x^2\|_t, 
}
where $K_3>0$ is a constant. The above estimates,
the first equality in \rf{eqs1111}, and expression \rf{psi2222} imply that
 there exists a constant $K_4>0$ such that for all $t\gt T$,
\aa{
\|\Gm(x^1) - \Gm(x^2)\|_t
\lt K_4 \int_t^{+\infty}  \|x^1-x^2\|_s \, ds.
}
From here we obtain that for $x^1,x^2\in X_T$,
\aa{
\|\hat \Gm(x^1) - \hat \Gm(x^2)\|_T
\lt K_4 \int_T^{+\infty}  \|x^1_\f-x^2_\f\|_s \, ds.
}
Suppose $x^n,x\in X_T$ and $\|x^n - x\|_T\to 0$. 
Then, for all $s\gt T$, $\|x^n_\f - x_\f\|_s \lt  \|x^n - x\|_s
\lt \|x^n - x\|_T \to 0$.
 On the other hand, $\|x^n_\f\|_s$ and $\|x_\f\|_s$ are bounded
 by $\frac{\kp}4 e^{-\te\cp s}$ on $[T,+\infty)$. This implies that
 $\|\hat \Gm(x^n) - \hat \Gm(x)\|_T\to 0$ by the dominated convergence
 theorem. Hence, the map $\hat \Gm: X_T\to X_T$ is continuous. 

{\it Step 4. Compactness of $\hat \Gm$.}
Let us show that $\hat  \Gm:  X_T  \to X_T$ is a compact map. 
Note that if $x\in\Xi_T$, by \rf{y1111}, $y(t) = \psi(t)O(\|\X\|_t)$.
A more careful consideration than \rf{gm2222} and the definition
\rf{gm1111} of the map $\Gm$ show that 
\aa{
\Gm(x)(t) = \te\cp \kp \int_t^{+\infty} e^{-\te \cp s}
\big( O(\|\X\|_t)(1+\X(s))^q -1 + o(1)\big) ds.
}
Note that  $o(1)$
in the latter formula expresses a function which tends
to zero uniformly in $x\in\Xi_T $. 
Indeed, it comes from the second expression in the first line 
in \rf{2nd1111} which is uniform in $x\in\Xi_T$.
Since $\hat \Gm(x) = \Gm(x_\f)$, 
there exists a constant $K>0$ such that
\aa{
|\hat \Gm(x)(t)|
\lt K \, e^{- \te \cp t} 
\quad  \text{for all} \;\;  x\in X_T \; \; \text{and} \;\; t\gt T.
}
It is suffices to prove that for any $\eps>0$,
for the family of functions $\{\hat \Gm (x)\}_{x\in X_T}$, 
there is a finite $\eps$-net.
Note that the family 
$\{\hat \Gm(x)\}_{x\in X_T}$ is uniformly bounded and equicontinuous. 

Let $\eps>0$ be arbitrary. Find $S>T$ such that
$K  e^{- \te \cp t}<\frac\eps2$ for $t\gt S$.
By the Arzel\`a-Ascoli theorem, find a finite $\frac\eps2$-net
for the family of functions  $\{\hat \Gm(x)\}_{x\in X_T}$
restricted to $[T,S]$; moreover, choose this $\frac\eps2$-net 
from the functions of the aforementioned family. Next, we extend
each function forming the $\frac\eps2$-net on $[T,S]$
to $[T,+\infty)$ in such a way
that its modulus is bounded by $K e^{- \te \cp t}$.
The extended functions form an $\eps$-net for the family 
$\{\hat \Gm(x)\}_{x\in X_T}$ considered on $[T,+\infty)$.
Hence, $\hat \Gm:  X_T  \to X_T$ is a compact map.

{\it Step 5. Existence of a fixed point of $\Gm$.}
By Schauder's fixed point therem (Theorem \ref{schauder}), the map
$\hat \Gm: X_T\to X_T$ has a fixed point $x^0\in X_T$. Let $\X^0$
be associated with $x^0$ by \rf{x1111}. 
Choose $S>T$ such that $\|\X^0\|_S\lt \frac14$. Then,
$x^0_\f(t) = x^0(t)$ for $t\gt S$.  
Hence, for all $t\gt S$, $\Gm(x^0)(t) = x^0(t)$. Without loss of generality, we can
set $S=T$ again, so $x^0$ is a fixed point of the map $\Gm$. 
The proof is now complete. 
\end{proof}
\begin{proposition}
\lb{pro99}
Assume \ref{A1}--\ref{A4}.
 Then, the solution $v^*$, constructed in Theorem \ref{t1111}, can be extended to 
a solution for \rf{w-int1111} on $(-\infty,+\infty)$.
Moreover, the extended solution $v^*$ is strictly increasing and 
there exists $T^*\in (-\infty,T)$ such that $v^*(T^*) = 0$.
\end{proposition}
\begin{proof}
The proof follows the scheme of the proof of Proposition 2.2 in \cite{jmss}. 
However, it naturally differs from the latter due to a different nonlinearity, so we 
present the proof here for the reader's convenience.
 
Equation \rf{w-int1111} immediately implies that ${v^*}'(t)>0$
for any  possible extension of $v^*$. 
Let $w(t) = e^{(\beta+1-\al)t} {v^*}'(t)^{\beta+1}$. We then obtain the following system with respect to $w$ and $v^*$:
\eq{
\lb{sys4}
w'(t) = -e^{-(\gm+1) t} \frac{h(t)}{g(v^*(t))},\\
{v^*}'(t) = e^{\l(\frac{\al}{\beta+1} -1\r)t} w^\frac1{\beta+1}.
}
Pick $S>T$ and define $c_1 = w(S)$. 
Integrating the first equation in \rf{sys4} gives
\aaa{
\lb{ts1111}
w(t) = w(S) +  \int_t^S  e^{-(\gm+1) s} \frac{h(s)}{g(v^*(s))} ds,
\quad t\lt S.
}
By \ref{A4}, on $(-\infty,S]$, 
$c_1 \lt w(t) \lt  c_1 + c_2 \, e^{-(\gm+1)t}= \mf c_t$ for some constant $c_2>0$.
Define $c_0 = v^*(S)$,
$\nu =  \ind_{(-\infty, c_0+\eps]} \ast \rho_\eps$,
and $\sig_t =  \ind_{[c_1-\eps,\mf c_t+\eps]}\ast \rho_\eps$, 
where $\rho_\eps$, $\eps<c_1$, is a standard mollifier supported
on the ball of radius $\eps$. Note that $\nu(\fdot) = 1$ on $(-\infty,c_0]$ and $\sig_t(\fdot) = 1$ on $[c_1, \mf c_t]$.
Instead of \rf{sys4}, consider  the system of first-order ODEs on $(-\infty,S]$
\eq{
\lb{30}
w'(t) = - e^{-(\gm+1)t} \frac{h(t)\nu(v^*)}{g(v^*)},\\ 
{v^*}'(t) = e^{\l(\frac{\al}{\beta+1} -1\r)t} (\sig_t(w) w)^\frac1{\beta+1}.
}
On $[T,S]$, $(v^*(t),w(t))$ is also a solution to \rf{30}, since over this interval $\nu(v^*(t)) = 1$ and $\sig_t(w(t)) =1$.
Further, since $\frac{h(t)\nu(v^*)}{g(v^*)}\lt \frac{\sup h(t)}{g(c_0+2\eps)}$
and  $0<(\sig_t(w) w)^\frac1{\beta+1} \lt (\mf c_t + 2\eps)^\frac1{\beta+1}$, 
one can extend $(v^*,w)$ to $(-\infty, S]$
(see, e.g., \cite{filippov}, Chapter 2, \S 6),
obtaining by this a solution $(\td v,\td w)$ to \rf{30} which coincides with $(v^*,w)$ on $[T,S]$. 
Since on $(-\infty,S]$, $\nu(\td v) = 1$ and $\sig_t(\td w) = 1$,
 $(\td v, \td w)$ is also a solution to \rf{sys4}.
Hence, we write $(v^*,w)$ (instead of $(\td v,\td w)$) for the extended solution.
There are two possible situations: either $v^*$ is strictly 
positive over $(-\infty,S]$, or there exists $T^*\in\Rnu$ such that $v^*>0$ on $(T^*,S]$ and $v^*(T^*) = 0$.
Let us show that the first situation  cannot be realized. 
If  $v^*$ is strictly  positive over $(-\infty,S]$, then there exists a finite limit
$L = \lim_{t\to-\infty} {v^*}(t)$. We use equation \rf{ts1111} with $t=-R$ and $S=0$,
where $R>0$ is sufficiently large. 
Since $\frac{h(t)}{g(v^*)}>\frac{\inf h(t)}{g(L)}$, we obtain that
\aa{
{v^*}'(-R)^{\beta+1} \gt e^{-(\al-\beta-1)R}{v^*}'(0)^{\beta+1}  
+\frac{\inf h(t)}{(\gm+1)g(L)}\big(e^{\te R}-e^{-(\al-\beta-1)R}\big).
}
Therefore, ${v^*}'(-R) \to +\infty$ as $R\to +\infty$. This implies that
${v^*}(-R) = {v^*}(0) - \int_{-R}^0 {v^*}'(t) dt \to -\infty$ as $R\to +\infty$. This contradicts to the fact that $\lim_{t\to-\infty} {v^*}(t)=L<+\infty$.
Thus, we conclude that the solution $v^*$ to  
 \rf{w-int1111} can be extended to $\Rnu$ in such a way 
that it crosses the $x$-axis at some point $T^*\in\Rnu$.
\end{proof}

\subsection{Proof of Theorem \ref{ss2222}}

\begin{proof}

{\it Step 1. Expressing $u^*(r)$ via $v^*(t)$.} Let $T^*$ be as in Proposition \ref{pro99}. 
Recall that  equation
\rf{w-dif1111} was obtained by the change of variable $r=\la^{-\frac1\te}e^{-t}$.
Since $v^*(T^*) = 0$, to satisfy the boundary condition $u^*(1) = 0$, 
we choose $\la^*$ in such a way that $(\la^*)^{-\frac1\te}e^{-T^*} = 1$, i.e.,
$\la^* = e^{-\te T^*}\!.$ The corresponding change of variable then becomes
$r=e^{T^*-t}$.
Therefore, $u^*(r) =v^*(T^*-\ln r)$ is
the touchdown solution, solving the equation in \rf{G} with $\la=\la^*$
and satisfying the boundary condition $u^*(1) = 0$.

{\it Step 2. Asymptotic behavior of $v^*(t)$.}
By \rf{v1111}, Taylor's formula, and Remark \ref{a2222}, 
\mm{
v^*(t) = g^{-1}(\kp e^{-\te\cp t}) 
+ x(t) \int_0^1\frac{d\xi}{g'\l(g^{-1}(\kp e^{-\te\cp t} + \xi x(t))\r)}
=g^{-1}(\kp e^{-\te\cp t}) - \frac{e^{-\te\cp t} o(1)}{Ak^qe^{-\te\kp q t}}\\
1 + \kp e^{-\te\cp t}\int_0^1 \frac{d\xi}{g'\l(g^{-1}(\xi\kp e^{-\te\cp t})\r)}
+ o(e^{-\te\cp (1-q) t})
= 1-\frac{\kp\, e^{-\te\cp t}}{A \kp^q e^{-\te\cp q t}}\int_0^1\frac{d\xi}{\xi^q}
+ o(e^{-\te\cp (1-q) t})  \\ = 1- \frac{\kp^{1-q}}{A(1-q)}e^{-\te\cp (1-q) t}
+ o(e^{-\te\cp (1-q) t}) \quad  (t\to+\infty).
}

{\it Step 3. Obtaining \rf{s-sol1111}.}
Since $e^{-t} = r e^{-T^*}$ and $\la^* = e^{-\te T^*}$, 
we obtain that $e^{-\te\cp (1-q) t} = r^{\te\cp (1-q)} (\la^*)^{\cp(1-q)}$.
This implies \rf{s-sol1111}.

{\it Step 4. $\la^*$ is the pull-in voltage.}  This follows immediately from Proposition
\ref{unique1111}.
\end{proof}
\begin{remark}
\rm Note that by uniqueness of a touchdown solution
(Proposition \ref{unique1111}), the solution $u^*$, constructed above,
is the unique integral touchdown solution to \rf{G}.
\end{remark}

\subsection*{Acknowledgements} 
 J.M.~do~\'O acknowledges partial support  from 
CNPq through the grants 312340/2021-4  and 
429285/2016-7 and 
Para\'iba State Research Foundation (FAPESQ), grant no 3034/2021.
 E. Shamarova acknowledges partial support from 
Universidade Federal da Para\'iba
 (PROPESQ/PRPG/UFPB) through the grant 
PIA13631-2020 (public calls no 03/2020 and 06/2021).


\begin{thebibliography}{99}
\bibitem{MR2500847} D. Castorina, P. Esposito, and B. Sciunzi. $p$-MEMS equation on a ball. Methods Appl. Anal., 15(3):277--283, 2008.
%
\bibitem{MR2471138}
D. Castorina, P. Esposito, and B. Sciunzi. Degenerate elliptic equations with singular nonlinearities. Calc. Var. Partial Differential Equations, 34(3):279--306, 2009. 
%
\bibitem{MR1422009}
P. Cl\'ement, D. G. de Figueiredo, and E. Mitidieri. Quasilinear elliptic equations with critical exponents. Topol. Methods Nonlinear Anal., 7(1):133--170, 1996.
%
\bibitem{MR0382848}
M. G. Crandall and P. H. Rabinowitz. Some continuation and variational methods for positive solutions of nonlinear elliptic eigenvalue problems. Arch. Ration. Mech. Anal., 58(3):207--218, 1975.
%
\bibitem{MR2569324}
J. D\'avila. Singular solutions of semi-linear elliptic problems. Handbook of differential equations: stationary partial differential equations. Vol. VI, Handb. Differ. Equ., 
83--176. Elsevier/North- Holland, Amsterdam, 2008.
%
\bibitem{MR2604963}
P. Esposito, N. Ghoussoub, and Y. Guo. Mathematical analysis of partial differential equations modeling electrostatic MEMS, volume 20 of Courant Lecture Notes in Mathematics. Courant
Inst. Math. Sci., New York; American Mathematical Society, Providence, RI, 2010.
%
\bibitem{filippov} A.F. Filippov, Introduction to the theory of differential equations (Vvedenie v teoriyu differenczialnih uravneniiy, in Russian), KomKniga, 2007.
%
\bibitem{Guo}
Y. Guo, Z. Pan, M.J. Ward, Touchdown and pull-in voltage behavior of a MEMS device with varying dielectric properties, SIAM J. Appl. Math. 66, 309--338, 2005.

\bibitem{Guo1}
 J.-S. Guo, P. Souplet, No touchdown at zero points of the permittivity profile for the MEMS problem, SIAM J. Math. Anal. 47, 614--625, 2015.
 
\bibitem{Guo2}
 Y. Guo, On the partial differential equations of electrostatic 
 MEMS devices III: refined touchdown behavior, J.
Differential Equations 244, 2277--2309, 2008.
\bibitem{GM}
M. Ghergu and Y. Miyamoto, Radial single point rupture solutions for a general MEMS model. Calc. Var. Partial Differential Equations, 61, 47, 2022.
%
\bibitem{MR1929156}
J. Jacobsen and K. Schmitt. The Liouville-Bratu-Gelfand problem for radial operators. J. Differential Equations, 184(1):283--298, 2002.
%
\bibitem{MR2166492}
J. Jacobsen and K. Schmitt. Radial solutions of quasilinear elliptic differential equations. Handbook of differential equations, pages 359--435. Elsevier/North-Holland, Amsterdam, 2004.
%
\bibitem{KMS}
N. Kavallaris, T. Miyasita, and T. Suzuki, 
Touchdown and related problems in electrostatic MEMS device equation,
NoDEA Nonlinear Differential Equations Appl. 15, 363--385, 2008.
%
\bibitem{Korman}
P. Korman, Infinitely many solutions for three classes of self-similar equations with $p$-Laplace operator: Gelfand, Joseph-Lundgren and MEMS problems, Proc. Roy. Soc. Edinburgh Sect. A 148, 341--356, 2018.
%
\bibitem{MR3759060}
J. M. do \'O  and R. G. Clemente. Some elliptic problems with singular nonlinearity and advection for Riemannian manifolds. J. Math. Anal. Appl., 460(2):582--609, 2018.
\bibitem{MR3491876}
J. M. do \'O and E. da Silva. Quasilinear elliptic equations with singular nonlinearity. Adv. Nonlinear Stud., 16(2):363--379, 2016.
\bibitem{MR3571913}
J. M. do \'O and E. da Silva.
Some results for a class of quasilinear elliptic equations with singular nonlinearity. Nonlinear Anal., 148:1--29, 2017.
\bibitem{jmss}
J.M. do \'O, E. Shamarova, and E. da Silva, Singular solutions to $k$-Hessian equations with fast-growing nonlinearities, Nonlinear Anal. 222, Paper no 113000, 2022.
%
\bibitem{MR1955412}
J. A. Pelesko and D. H. Bernstein. Modeling MEMS and NEMS. Chapman \& Hall/CRC Math., Boca Raton, FL, 2003.
\bibitem{smart} D.R. Smart. Fixed point theorems.
Cambridge University Press, 1980.
\end{thebibliography}
\end{document}